\documentclass[11pt]{amsart}

\usepackage[utf8]{inputenc}
\usepackage[english]{babel}
\usepackage[T1]{fontenc}
\usepackage{amsthm}
\usepackage{amssymb}
\usepackage{amsmath}
\usepackage{mathtools}
\usepackage{pdfpages}
\usepackage{tikz-cd}
\usepackage{soul}
\usepackage{mathabx}

\newcommand{\PP}{\mathcal{P}_T}

\newcommand{\N}{\mathbb{N}}
\newcommand{\Z}{\mathbb{Z}}

\newcommand{\GG}{\mathcal{G}}

\newcommand{\Aut}{\operatorname{Aut}}

\newcommand{\id}{\mathrm{I}}
\newcommand{\Id}{\mathrm{Id}}
\newcommand{\Soc}{\mathrm{Soc}}

\makeatletter
\numberwithin{equation}{section}
\numberwithin{figure}{section}
\numberwithin{table}{section}
\newtheorem{thm}{Theorem}[section]
\newtheorem{lem}[thm]{Lemma}
\newtheorem{cor}[thm]{Corollary}
\newtheorem{pro}[thm]{Proposition}
\newtheorem{defn}[thm]{Definition}

\newtheorem{rem}[thm]{Remark}
\newtheorem{exa}[thm]{Example}

\newtheorem{thmA}{Theorem}

\@namedef{subjclassname@2020}{%
	\textup{2020} Mathematics Subject Classification}

\makeatother

\begin{document}
	
	
	
	\title[Cabling, decomposability, and Dehornoy class]{Involutive Yang--Baxter: cabling, decomposability,  Dehornoy class}
	
	\author[V. Lebed]{Victoria Lebed}
	\address{Normandie Univ, UNICAEN, CNRS, LMNO, 14000 Caen, France}
	\email{victoria.lebed@unicaen.fr}
	
	\author[S. Ram\'irez]{Santiago Ram\'irez}
	\address{IMAS--CONICET and Depto. de Matem\'atica, FCEN, Universidad de Buenos Aires, Pabell\'on~1,
		Ciudad Universitaria, C1428EGA, Buenos Aires, Argentina}
	\email{sramirez@dm.uba.ar}
	\email{lvendramin@dm.uba.ar}
	
	\author[L. Vendramin]{Leandro Vendramin}
	\address{Department of Mathematics and Data Science, Vrije Universiteit Brussel, Pleinlaan 2, 1050 Brussel, Belgium}
	\email{Leandro.Vendramin@vub.be}
	
	\date{\today}
	
	\begin{abstract}
		We develop new machinery for producing decomposability tests for involutive solutions to the Yang--Baxter equation. It is based on the seminal decomposability theorem of Rump, and on ``cabling'' operations on solutions and their effect on the diagonal map~$T$. Our machinery yields an elementary proof of a recent decomposability theorem of Camp--Mora and Sastriques, as well as original decomposability results. It also 
		provides a conceptual interpretation (using the language of braces) of the Dehornoy class, a combinatorial invariant naturally appearing in the Garside-theoretic approach to involutive solutions. 
	\end{abstract}
	
	\maketitle
	
	\section{Introduction}
	
	A \emph{finite non-degenerate involutive set-theoretic solution to the Yang--Baxter equation}, simply called \emph{solution} in this paper, is a non-empty finite set $X$ endowed with an involutive map
	\begin{gather*}
	r\colon X\times X\to X\times X,
	\quad
	r(x,y)=(\sigma_x(y),\tau_y(x)), 
	\shortintertext{satisfying}
	r_1 r_2 r_1=r_2 r_1 r_2,
	\end{gather*}
	where the maps $r_i \colon X^3 \to X^3$ are defined by $r_1=r \times \Id_X$ and $r_2 = \Id_X \times r$, and the maps $\sigma_x$ and $\tau_x$ are bijective for all $x\in X$. Throughout the paper, we will assume that $n=|X|>1$. The origins, applications and recent results on solutions can be found in the extensive literature which followed \cite{MR1722951,MR1637256}.
	
	A solution is called \emph{decomposable} if the set $X$ decomposes into two non-empty disjoint parts $X=Y \sqcup Z$, with $r(Y\times Y) = Y\times Y$ and $r(Z\times Z) = Z\times Z$. This is equivalent to asking the \emph{permutation group} $\GG(X,r)$ of $(X,r)$, which is the group of permutations on~$X$ generated by all the $\sigma_x$, 
	 to be non-transitive \cite{MR1722951}. A natural approach to the (as for now unattainable) problem of classifying all solutions consists in constructing all indecomposable solutions \cite{MR3771874,MR3958100,MR4116644,MR4163866,MR4085769,MR4300920,MR4391683,MR4388351,C,MR4396637}, and then understanding how these building blocks can cement together \cite{MR2442072,MR3812099,MR4020748,MR4009573,MR4067191}.
	
	The first and most famous result on decomposability is the 1996 conjecture of Gateva-Ivanova, proved by Rump in 2005 \cite{MR2132760}: a \emph{square-free} solution (i.e., satisfying $r(x,x)=(x,x)$ for all $x$) is decomposable. 
	 In general, a solution needs not be square-free; however, the \emph{diagonal map} \[T \colon x \mapsto \tau_x^{-1}(x)\] always defines a permutation on $X$, satisfying \[r(T(x),x)=(T(x),x).\] 
	The permutation $T$ splits $X$ into orbits. This induces a partition of $n=|X|$, which we call the \emph{$T$-partition} of $(X,r)$ and denote by $\PP = \PP(X,r)$. 
	Two recent 
	papers \cite{RV,CMS} revealed that this simple numerical datum may suffice to determine the (in)decomposability of a solution: 
	\begin{itemize}
		\item $\PP =(1,\ldots,1)$ $\;\Rightarrow\;$ $(X,r)$ decomp. (Rump's theorem, \cite[Thm 1]{MR2132760});
		\item $\PP =(n-1,1)$ $\;\Rightarrow\;$ $(X,r)$ decomp. (\cite[Thm 3.10]{RV});
		\item $\PP =(n-2,1,1)$, $n$ odd $\;\Rightarrow\;$ $(X,r)$ decomp. (\cite[Thm 3.13]{RV});
		\item $\PP =(n-3,1,1,1)$, $3 \notdivides n$ $\;\Rightarrow\;$ $(X,r)$ decomp. (\cite[Thm 3.13]{RV});
		\item more generally, $\gcd(|T|,n)=1$ $\;\Rightarrow\;$ $(X,r)$  decomp. (Camp-Mora--Sastriques (CMS) theorem, \cite[Thm A]{CMS}); 
		\item $\PP=(n)$ $\;\Rightarrow\;$ $(X,r)$ indecomp. (\cite[Thm 3.5]{RV}).
	\end{itemize}
	
	In this paper, we give a short and elementary proof of the CMS theorem (which was originally proved using advanced group theory), and present several original decomposability results. To explain them, we need the \emph{structure group} of our solution $(X,r)$ \cite{MR1722951}. It is defined by the following presentation:
	\[G{(X,r)} = \langle X \,|\, xy = \sigma_x(y)\tau_y(x) \text{ for all } x,y \in X \rangle.\]
	It carries a second, commutative operation $+$, satisfying the following compatibility relation: 
	\begin{equation}\label{E:Brace}
	a(b+c)=ab-a+ac.
	\end{equation}
	Such ring-like structures are called \emph{braces}. They are extensively used in order to bring ring-theoretic tools into the study of the YBE: see \cite{MR2278047} and references thereto. In the present work, we will employ braces in a very different way. Namely, given a positive integer $k$, consider the map
	\begin{align*}
	\iota^{(k)} \colon X &\to G{(X,r)},\\
	x &\mapsto kx := x+\cdots + x \ (k \text{ summands}).
	\end{align*}
	
	\begin{thmA}\label{T:cabling}
		Take a solution $(X,r)$ and a positive integer $k$. The map $\iota^{(k)}$ above is injective. Its image is a sub-solution of $G{(X,r)}$.
	\end{thmA}
	Here the solution $r$ is extended from $X$ to $G{(X,r)}$ in the usual way. A push-back through $\iota^{(k)}$ then defines a new solution on~$X$, called the \emph{$k$-cabled solution} $r^{(k)}$. Equation~\eqref{E:cabling_r} describes it explicitly. Some relations between a solution and its cablings are given in 
	
	\newpage 
	\begin{thmA}\label{T:cablingProperties}
		Take a solution $(X,r)$ and a positive integer $k$. 
		\begin{enumerate}
			\item The diagonal map of $(X,r^{(k)})$ is $T^k$.
			\item Let $x\in X$ lie in a $\GG{(X,r)}$-orbit of size $m$, and in a $\GG{(X,r^{(k)})}$-orbit of size $m'$. Then $m'$ is a multiple of the maximal divisor $m_k$ of $m$ which is coprime to $k$: $m_k \ | \ m$ and  $\gcd(m_k, k) = 1$.
		\end{enumerate}
	\end{thmA}
	
	In particular, if the solution $(X,r)$ is indecomposable and $\gcd(|X|,k)=1$, then $(X,r^{(k)})$ remains indecomposable (since $|X|_k=|X|$). Taking $k=|T|$, we then reduce the CMS theorem to Rump's result. On the other hand, taking $k=p$, $k=a$, and $k=2$ respectively, we obtain new decomposability theorems:
	
	\begin{thmA}\label{T:pq}
		Take an indecomposable solution $(X,r)$ of size $pq$, where $p \neq q$ are prime. Then its $T$-partition cannot contain a term $s$ satisfying $(p-1)q < s < pq$ and $\gcd(s,p)=1$.
	\end{thmA}
	For instance, for $|X|=14$, this excludes cycles of size $9,11$ and $13$, and for $|X|=15$, this excludes cycles of size $11,13$ and $14$.
	
	\begin{thmA}\label{T:ab}
		Take an indecomposable solution $(X,r)$ of size $ab$ and $T$-partition $(a,c,c')$, where the numbers $a,b,c,c'$ are pairwise coprime, except for, possibly, $c$ and $c'$. Then one cannot have $b>a+c$.
	\end{thmA}
	As a consequence, indecomposable solutions $(X,r)$ of size $2b$ with odd $b \geq 5$ cannot have $T$-partition $(2,b-4,b+2)$.
	
	\begin{thmA}\label{T:2abc}
Take an indecomposable solution $(X,r)$ of size $2d$, with $d$ odd, and $T$-partition $(2a,b,c)$, where $\gcd(2d, abc)=1$ and $b \leq c$. Then $2a+b=c$.
	\end{thmA}
	For instance, for $|X|=18$ this excludes the $T$-partition $(10,7,1)$, and for $|X|=22$ this excludes the $T$-partitions $(10,9,3)$, $(10,7,5)$, and $(6,7,9)$. Neither of these are covered by Theorems \ref{T:pq} and \ref{T:ab}.
	
	In parallel with its decomposition into $T$-cycles, a solution carries several other relevant decompositions: into imprimitivity blocks, and into $\GG{(X,r^{(k)})}$-orbits (for well chosen $k$). Comparing them, and using the recent classification of primitive solutions from \cite{CJO_primitive}, we obtain

	\begin{thmA}\label{T:pq2}
	Take an indecomposable solution $(X,r)$ of size $pq$, with $p<q$ prime. Then its $T$-partition contains either only multiples of $q$, or at least one multiple of~$p$.
	\end{thmA}

	In particular, for an indecomposable solution of size $n=2q$ with an odd prime $q$, the only possible $T$-partition with only odd terms is $(q,q)$.
	
	More generally, Theorem~\ref{T:cablingProperties} allows one to considerably reduce the list of possible $T$-partitions for indecomposable solutions. This has the potential to speed up algorithms constructing all indecomposable solutions of small size.
	
	In another vein, cabling can produce new indecomposable solutions out of old ones: see Example~\ref{EX:NewIndec}. 
	
	In the final part of the paper, cabling and brace ideas are used to explore an important invariant of a solution $(X,r)$, which we propose to call its \emph{Dehornoy class}. It is the smallest positive integer $m$ such that
	\begin{equation}\label{E:class}
	\forall x \in X, \; \sigma_{T^{m-1}(x)} \cdots \sigma_{T(x)}\sigma_x=\Id.
	\end{equation}
	Such an~$m$ always exists, and is $ < (n^2)!$. The elements $mx$, $x \in X$, then generate a normal free abelian subgroup of $G{(X,r)}$ of finite index. The corresponding finite quotient plays the same role as Coxeter groups play for Artin groups. In particular it suffices for reconstructing the Garside structure on the whole $G{(X,r)}$. For details, see \cite{MR3190423, MR3374524}. A partial generalisation to non-involutive solutions is proposed in \cite{MR3974961}. 
	
	Now, the permutation group $\GG{(X,r)}$ inherits the brace structure from $G{(X,r)}$ \cite{MR3177933}. We then give a new conceptual interpretation of the Dehornoy class in terms of the abelian group $(\GG(X,r),+)$:
	
	\begin{thmA}\label{T:DClass}
		The Dehornoy class $m$ of a solution $(X,r)$ is the least common multiple of the orders of the generators $\sigma_x$, $x \in X$, of the group $(\GG(X,r),+)$. If the solution is indecomposable, $m$ is the order of any $\sigma_x$.
	\end{thmA}

	Another type of problems where cabling can be useful is the structural study of braces. Since these questions are out of the focus of the present work, we simply illustrate this approach with a quick proof of two important properties of finite braces at the end of Section~\ref{S:cablingProperties}.
	
	\section{Cabling a solution}\label{S:cabling}
	
	In this section we prove Theorem~\ref{T:cabling}.
	
	Recall that, with respect to the operation~$+$, $G{(X,r)}$ is a free abelian group, and the elements $x \in X$ yield its basis \cite{MR2278047}. Therefore the map $\iota^{(k)}$ is injective.
	
	To get the second assertion of the theorem, we will prove an explicit formula for the extension $R$ of $r$ to $G{(X,r)}$:
	\begin{equation}\label{E:cabling_R}
	R(kx,ly)= (l\sigma_{kx}(y),k T^{k-1}\tau_{ly}T^{-k+1}(x)),
	\end{equation}
	where $k$ and $l$ are positive integers. This yields
	\begin{equation}\label{E:cabling_r}
	r^{(k)}(x,y)= (\sigma_{kx}(y), T^{k-1}\tau_{ky}T^{-k+1}(x)),
	\end{equation}
	and finishes the proof of Theorem~\ref{T:cabling}.
	
	Recall that the operation~$+$ on $G{(X,r)}$ is a natural extension of the law
	\[x+y = x \sigma_x^{-1}(y), \qquad x,y \in X.\]
	In particular,
	\begin{equation}\label{E:kx}
	kx=x U(x) U^2(x) \cdots U^{k-1}(x),
	\end{equation}
	where \[U \colon x \mapsto \sigma_x^{-1}(x)\] is the inverse of the diagonal map~$T$. One recognises the frozen words from~\cite{MR3190423} (for $k=2$), and the twisted powers $x^{[k]}$ from~\cite{MR3374524} (for general~$k$).
	
	Let us look at 
	\begin{equation}\label{E:frozen}
	r_{k,l}(\ x U(x) U^2(x) \cdots U^{k-1}(x) \ ,\ y U(y) U^2(y) \cdots U^{l-1}(y)\ )=(\overline{u},\overline{w}),
	\end{equation}
	where the tuples $(x_1,\ldots,x_s)$ are denoted by $x_1 \cdots x_s$ for simplicity, and the solution $r$ is extended, this time, to the powers of $X$:
\[r_{k,l} = (r_{l}\cdots r_{1}) \cdots (r_{k+l-2}\cdots r_{k-1}) (r_{k+l-1}\cdots r_k) \colon X^k \times X^l \to X^l \times X^k.\]
These maps induce the solution $R$ on $G{(X,r)}$, as explained in \cite{MR2383056} with an inductive argument, and in \cite{MR3558231} with a graphical argument. Both entries in~\eqref{E:frozen} are \emph{frozen} tuples, that is, they remain unchanged when $r$ is applied to any neighbouring positions, since $r(z,U(z)) =(z,U(z))$ for all $z \in X$. But the YBE for $r$ guarantees that
	\[r_i r_{k,l} = \begin{cases}
	r_{k,l} r_{k+i} & \text{ if } 1 \leq i \leq l-1,\\
	r_{k,l} r_{i-l} & \text{ if } l+1 \leq i \leq l+k-1.
	\end{cases}\]
	Here $r_i$ is the solution $r$ applied at the positions $i$ and $i+1$ of a tuple. As a result, $\overline{u}$ and $\overline{w}$ are also frozen:
	\[\overline{u} = y' U(y') \cdots U^{l-1}(y')=ly', \qquad \overline{w} = x' U(x') \cdots U^{k-1}(x')=kx'.\]
	Thus $R(kx,ly)= (ly',kx')$. Further, \eqref{E:frozen} implies 
	\begin{align*}
	y'&= \sigma_{x U(x) U^2(x) \cdots U^{k-1}(x)}(y) = \sigma_{kx}(y),\\
	U^{k-1}(x') &= \tau_{y U(y) U^2(y) \cdots U^{l-1}(y)}U^{k-1}(x) = \tau_{ly}U^{k-1}(x),
	\end{align*}
	hence $x'=U^{-k+1}\tau_{ly}U^{k-1}(x) = T^{k-1}\tau_{ly}T^{-k+1}(x)$, as announced.
	
	\section{Properties of cabled solutions}\label{S:cablingProperties}
	
	In this section we first prove Theorem~\ref{T:cablingProperties} and then turn to other properties of cabled solutions.
	
	\begin{proof}[Proof of Theorem~\ref{T:cablingProperties}]
		For all positive integer $k$ and $x \in X$, the tuple 
		\[x U(x) U^2(x) \cdots U^{2k-1}(x) \in X^{2k}\] is frozen. Since applying the solution $r_{k,k}$ to a $2k$-tuple boils down to applying $r$ repeatedly at different positions, one gets
		\begin{align*}
		&r_{k,k}(x U(x) U^2(x) \cdots U^{k-1}(x),U^{k}(x) U^{k+1}(x)  \cdots U^{2k-1}(x)) \\
		&=(x U(x) U^2(x) \cdots U^{k-1}(x),U^{k}(x) U^{k+1}(x) \cdots U^{2k-1}(x).
		\end{align*} 
		In other words, $R(kx,kU^{k}(x)) = (kx,kU^{k}(x))$, that is, \[
		r^{(k)}(x,U^{k}(x))=(x,U^{k}(x)).
		\]
		Since $T$ is the inverse of $U$, this yields $r^{(k)}(T^{k}(x),x)=(T^{k}(x),x)$. Therefore $T^{k}$ is the diagonal map for $r^{(k)}$.
		
		Now,  let $x\in X$ lie in the $\GG{(X,r)}$-orbit of size $m$ and in the $\GG{(X,r^{(k)})}$-orbit of size $m'$. Denote by $\GG{(X,r)}_x$ and $\GG{(X,r^{(k)})}_x$ its stabilisers in the two groups. One has
		\begin{align*}
		|\GG{(X,r)}| &= m |\GG{(X,r)}_x|, & |\GG{(X,r^{(k)})}| &= m' |\GG{(X,r^{(k)})_x}|.
		\end{align*}
		
		The permutation groups $\GG{(X,r)}$ and $\GG{(X,r^{(k)})}$ inherit brace structures from the corresponding structure groups $G{(X,r)}$ and $G{(X,r^{(k)})}$ (see \cite{MR3177933}). Moreover, the abelian group $(\GG{(X,r^{(k)})},+)$ is obtained from the abelian group $(\GG{(X,r)},+)$ by multiplying each of its generators $\sigma_x$, $x \in X$, by $k$. Thus its size is the size of $(\GG{(X,r)},+)$ divided by some product $p_1^{d_1} \cdots p_l^{d_l}$ of powers of prime divisors of~$k$. Also, since the permutation group $\GG{(X,r^{(k)})}$ is the subgroup of $\GG{(X,r)}$ generated by the permutations $\sigma_{kx}$, $x \in X$, the stabiliser $\GG{(X,r^{(k)})}_x$ is a subgroup of $\GG{(X,r)}_x$. Hence $|\GG{(X,r^{(k)})}_x| = |\GG{(X,r)}_x| \ / \ t$ for some positive integer $t$. Summarising, we obtain 
		\begin{align*}
		m' &= {|\GG{(X,r^{(k)})}|} \ / \ {|\GG{(X,r^{(k)})}_x|} = \frac{|\GG{(X,r)}|}{p_1^{d_1} \cdots p_l^{d_l}}  \ / \ \frac{|\GG{(X,r)}_x|}{t} = \frac{m t}{p_1^{d_1} \cdots p_l^{d_l}}.
		\end{align*}
		Since this fraction is an integer, it is a multiple of an integer of the form $\dfrac{m}{p_1^{d'_1} \cdots p_l^{d'_l}}$. Recalling that the $p_i$ are prime divisors of~$k$, we see that $\frac{m}{p_1^{d'_1} \cdots p_l^{d'_l}}$ is a multiple of the maximal divisor $m_k$ of $m$ which is coprime to $k$, as announced.
	\end{proof}
	
	\begin{pro}\label{P:cablingIterate}
		The iteration of cablings remains a cabling. More precisely, given a solution $(X,r)$ and positive integers $k$ and $k'$, one has
		\[(r^{(k)})^{(k')} = r^{(kk')}.\]
	\end{pro}
	
	\begin{proof}
		Formula \eqref{E:cabling_r} implies
		\begin{align*}
		(r^{(k)})^{(k')} (x,y) &= (\sigma_{k'kx}(y),\cdot),& r^{(kk')}(x,y)&= (\sigma_{kk'x}(y),\cdot).
		\end{align*}
		Recall the relation
		\begin{equation}\label{E:sigma_tau}
		T \sigma_x=\tau_x^{-1}T
		\end{equation}
		connecting $\sigma$'s and $\tau$'s (see for instance \cite[Proposition 2.2]{MR1722951}). It implies that the $\sigma$-component uniquely determines a solution. We are done.
	\end{proof}
	
	Recall that a solution $(X,r)$ is called \emph{retractable} if for some $x \neq x' \in X$, one has $\sigma_x=\sigma_{x'}$ (and hence $\tau_x=\tau_{x'}$). Identifying all such $x$ and $x'$, one gets the \emph{retraction} $\operatorname{Ret}(X,r)$ of $(X,r)$; it is a solution again, as explained in \cite{MR1722951}. This is an important property of solutions: see \cite{MR2885602} and references thereto. 
	
	\begin{pro}\label{P:cablingRetract}
		If a solution $(X,r)$ is retractable, then so are all its cablings. More precisely, $\operatorname{Ret}(X,r^{(k)})$ is a quotient of $\operatorname{Ret}(X,r)^{(k)}$ for all positive integers $k$.
	\end{pro}
	
	\begin{proof}
		Using the brace structure $\GG{(X,r)}$ inherits from $G{(X,r)}$, we can write $\sigma_{kx}=k\sigma_x$. Thus the relation $\sigma_x=\sigma_{x'}$ implies $\sigma_{kx}=\sigma_{kx'}$. From \eqref{E:cabling_r}, one then concludes that elements $x$ and $x'$ identified in $\operatorname{Ret}(X,r)$ are necessarily identified in $\operatorname{Ret}(X,r^{(k)})$ as well.
	\end{proof}
	
	Until now, all connections between solutions and braces that we used went through the brace structures on the structure and permutation groups of a solution. But one can go the other way round, and define a solution on any brace \cite{MR2278047}. This gives one the intuition on how to cable a brace. Concretely, take a brace $(B,+,\circ)$ and a positive integer $k$. The elements $ka$, $a \in B$, form a sub-brace $B^{(k)}$ of $B$, called its \emph{$k$-cabling}. Indeed, we have
	\begin{align*}
	ka+kb &= k(a+b),\\
	ka \circ kb &= k( (ka)\circ b -(k-1)a),
	\end{align*}  
	as follows from the commutativity of $+$ and from relation~\eqref{E:Brace} respectively. The additive structure of $B^{(k)}$ is obtained from $(B,+)$ by multiplication by $k$. One can thus easily determine its size. The multiplicative group $(B,\circ)$ then has a subgroup of the same size. Here are two direct applications:	
	\begin{enumerate}
	\item  A quick proof of the solvability of the multiplicative group of a finite brace (first established in \cite[Theorem 2.15]{MR1722951}). Indeed, let $(B,+,\circ)$ be a brace of size $ab$ with $\gcd(a,b)=1$. Looking at the additive structure, one sees that $B^{(a)}$ is of size $b$. Therefore $(B^{(a)},\circ)$ is a $b$-Hall subgroup of $(B,\circ)$. Thus $(B,\circ)$ is solvable. 
	\item Let $B$ be a finite brace with cyclic additive group, and $d$ a divisor of its size $|B|$. Then $(B,\circ)$ contains a subgroup of size $d$. Indeed, looking at the additive structure and using the cyclicity of $(B,+)$, one sees that $B^{(|B|/d)}$ is of size $d$.
	\end{enumerate}

	\section{Applications: (in)decomposability results}\label{S:cablingApplications}
	
	We now turn to applications of Theorem~\ref{T:cablingProperties}. Its assertion is particularly transparent when the solution $(X,r)$ is indecomposable, and the cabling parameter $k$ is coprime to its size $|X|$, which is now the size of the only $\GG{(X,r)}$-orbit. Since $|X|_k=|X|$, the theorem implies that the solution $(X,r^{(k)})$ remains indecomposable, with diagonal map $T^k$. Here are some interesting particular cases.
	
	\begin{enumerate}
		\item If $\gcd(|T|,|X|)=1$, then $(X,r^{(|T|)})$ has to be indecomposable, with diagonal map $\Id$, which is impossible by Rump's theorem. We thus recover the Camp-Mora--Sastriques (CMS) theorem.	
		\item If the cycle decompositions of $T$ and $T^k$ are different, we get a new indecomposable solution on the same set~$X$. For instance, if $(X,r)$ is the indecomposable solution with $T$-partition $(2,6)$ (cf. \cite[Table 3.2]{RV}), we have $|X|=2+6=8$, which is comprime with $k=3$. Then $(X,r^{(3)})$ is an indecomposable solution with $T$-partition $(2,2,2,2)$, and hence not isomorphic to $(X,r)$.
	\end{enumerate}
	
	To treat other cases, we need the following elementary observation.
	
	\begin{lem}\label{L:TandGorbits}
		Given a solution $(X,r)$ with diagonal map~$T$, any $T$-orbit in~$X$ lies entirely within a single $\GG{(X,r)}$-orbit. 
	\end{lem}
	
	\begin{proof}
		Take an element $x \in X$ from a $\GG{(X,r)}$-orbit $Y$. By~\cite{MR1722951}, $r$ restricts to $Y \times Y$ and defines a solution on~$Y$. The diagonal map of this restricted solution has to be the restriction of~$T$ to~$Y$. Thus the $T$-orbit of~$x$ lies entirely within~$Y$.
	\end{proof}	
	
	Now, take an indecomposable solution $(X,r)$ and a cabling parameter $k$ which is not coprime to $|X|$, but which makes $|X|_k$ big enough. Then the sizes of all  $\GG{(X,r^{(k)})}$-orbits are multiples of $|X|_k$. On the other hand, by the above lemma, all the $T^k$-orbits lie entirely inside these $\GG{(X,r^{(k)})}$-orbits. In several cases, for numerical reasons, this can happen only when there is only one $\GG{(X,r^{(k)})}$-orbit. The solution $(X,r^{(k)})$ is then indecomposable, which imposes some constraints on the sizes of the $T^k$-orbits, for instance by the CMS theorem. This leads to a contradiction in various cases which are not themselves covered by the CMS theorem. Here are some of them.
	
	\begin{enumerate}
		\item Take an indecomposable solution $(X,r)$ of size $pq$, where $p \neq q$ are primes. Assume that a $T$-orbit is of size $(p-1)q < s < pq$, with $\gcd(s,p)=1$. We will show that this is impossible, and thus prove Theorem~\ref{T:pq}. For any $t \in \N$, the diagonal map $T^{p^t}$ of $(X,r^{(p^t)})$ inherits this orbit, since $\gcd(s,p)=1$. Thus this $T^{p^t}$-orbit of size $s$ lies entirely within a $\GG{(X,r^{(p^t)})}$-orbit, whose size is a multiple of $|X|_{p^t}=q$. Since $(p-1)q < s < pq$, this $\GG{(X,r^{(p^t)})}$-orbit has to be the whole set~$X$. In other words, the $p^t$-cabled solution $(X,r^{(p^t)})$ is indecomposable. But, for $t$ big enough, the sizes of all $T^{p^t}$-orbits are coprime to $p$. But they are also coprime to $q$ since there is one orbit of size $(p-1)q < s < pq$ and several smaller orbits of total size $pq-s < q$. As a consequence, $\gcd(|X|,|T^{p^t}|) = \gcd(pq,|T^{p^t}|) =1$. By the CMS theorem, the solution $(X,r^{(p^t)})$ is then decomposable, contradiction. 
		\item Take an indecomposable solution $(X,r)$ of size $ab$ and $T$-partition $(a,c,c')$, where $b>a+c$, and the numbers $a,b,c,c'$ are pairwise coprime, except for, possibly, $c$ and $c'$. We will show that this is impossible, and thus prove Theorem~\ref{T:ab}. The $a$-cabling of $(X,r)$ has $T$-partition $(c,c',1, \ldots,1)$, with $a$ ones. Since  $\gcd(|X|,|T^{a}|)$ divides $\gcd(ab,cc')=1$, the CMS theorem says that $(X,r^{(a)})$ is decomposable, and there are at least two $\GG(X,r^{(a)})$-orbits. One of them does not contain the $T^{a}$-orbit of size $c'$, hence its size is $\leq c + a <b$, which is impossible for a multiple of $|X|_a=|ab|_a=b$.
		\item Take an indecomposable solution $(X,r)$ of size $2d$, with $d$ odd, and $T$-partition $(2a,b,c)$, where $\gcd(2d, abc)=1$ and $b\leq c$. We will show that it imposes heavy restrictions on $a,b,c$, and thus prove Theorem~\ref{T:2abc}. The $2$-cabling of $(X,r)$ has $T$-partition $(a,a,b,c)$, since $b$ and $c$ are odd. The sizes of its $\GG{(X,r^{(2)})}$-orbits are multiples of $(2d)_2=d$, as $d$ is odd. Since  $\gcd(|X|,|T^{2}|)$ divides $\gcd(2d,abc)=1$, the CMS theorem says that $(X,r^{(2)})$ is decomposable, so there are precisely two $\GG{(X,r^{(2)})}$-orbits, each of size~$d$. Each of the four $T^2$-orbits lies entirely in one of these two $\GG{(X,r^{(2)})}$-orbits. Since the numbers $a,b,c,d$ are all odd, this is possible only if $d=2a+b=c$.
		\item Assume that $(X,r)$ is an indecomposable solution of size $30$. We will show that its $T$-partition cannot be $(21,7,1,1)$. Indeed, the $3$-cabled solution $(X,r^{(3)})$ would then have $T$-partition $(7,7,7,7,1,1)$, and be decomposable by the CMS theorem. On the other hand, its $\GG(X,r^{(3)})$-orbits are multiples of $30_3=10$. But the only way to divide the multiset $(7,7,7,7,1,1)$ into parts whose total sums are all divisible by~$10$ is to take the whole multiset. Thus the solution $(X,r^{(3)})$ is indecomposable, contradiction. As in the above situations, this example generalises to an infinite family.
	\end{enumerate}
	
	In another vein, cabling can produce new indecomposable solutions out of old ones.
	
	\begin{exa}\label{EX:NewIndec}
		Consider the indecomposable solution $(X,r)$ (found by a computer), where $X=\{1,\dots,8\}$, $r(x,y)=(\sigma_x(y),\sigma^{-1}_{\sigma_x(y)}(x))$, and
		\begin{align*}
		&\sigma_1=(12)(34)(56)(78), 
		&& \sigma_2=(12)(36)(47)(58),
		&& \sigma_3=(1543)(2678),\\
		& \sigma_4=(1367)(2854),
		&&\sigma_5=(17)(24)(38)(56),
		&&\sigma_6=(1763)(2458),\\
		&\sigma_7=(1345)(2876),
		&&\sigma_8=(15)(26)(38)(47).
		\end{align*}
		Its diagonal map is $T=(12)(345678)$, so its $T$-partition is $(2,6)$. According to Theorem~\ref{T:cablingProperties}, its $3$-cabling $(X,r^{(3)})$ is still indecomposable (as $\gcd(3,8)=1$) and has $T$-partition $(2,2,2,2)$. It is thus not isomorphic to $(X,r)$.
	\end{exa}

	\section{Primitivity and further (in)decomposability results}\label{S:Primitivity}
	
	A solution $(X,r)$ is called \emph{imprimitive} if the $\GG(X,r)$-action on~$X$ is so, and \emph{primitive} otherwise. That is, an imprimitive solution $X$ admits a non-trivial decomposition into blocks which is preserved by the $\GG(X,r)$-action. A recent result from \cite{CJO_primitive} asserts that, up to isomorphism, the only primitive solutions are the permutation solutions $(\Z/p\Z, \  r(a,b)=(b-1,a+1))$, with $p$ prime. By \cite{MR1722951}, these are the only indecomposable solutions of prime size. Thus, in the interesting case of non-prime size, an indecomposable solution can be split into imprimitivity blocks. Their interaction with $T$-cycles is quite intricate. We will now analyse this interaction in the particular settings of Theorem~\ref{T:pq2}, and deduce a proof of that theorem.
	
	Consider an indecomposable solution $(X,r)$ of size $pq$, with primes $p<q$. Assume that its $T$-partition contains no multiples of~$p$, and at least one term which is not a multiple of $q$. We will obtain a contradiction, proving Theorem~\ref{T:pq2}.
	
	By Theorem~\ref{T:cablingProperties}, one can choose a suitable $k$ coprime with $pq$ such that the solution $(X,r^{(k)})$ is still indecomposable, has $T$-partition with all terms of the form $p^\alpha q^\beta$, and permutation group $\GG(X,r^{(k)})$ of size $p^a q^b$. (For the latter property, recall that the $k$-cabling multiplies all the elements of $(\GG(X,r),+)$ by $k$.) Since the cabling can only split $T$-orbits into equal parts, the $T$-partition of $(X,r^{(k)})$ still contains no multiples of~$p$, and at least one term which is not a multiple of $q$. Thus it suffices to work with solutions having these properties.
	
	Summarizing all the constraints on the $T$-partition we obtained, one sees that it has to be of the form $(q,\ldots,q,1,\ldots,1)$, with at least one term $1$ and one term $q$ (otherwise Rump's theorem applies). 

	Since $pq$ is not prime, our solution is imprimitive. Thus $X$ non-trivially decomposes into blocks preserved by the $\GG(X,r)$-action. Since $(X,r)$ is indecomposable, $\GG(X,r)$ permutes these blocks in a transitive manner, hence they are all of the same size. This leaves us with two possibilities.
	
	\medskip
	\noindent \textbf{Case 1:} There are $p$ blocks of size $q$.
	
	\noindent Since our solution is indecomposable, some map $\sigma_x$ permutes $1<p' \leq p$ blocks in a cyclic manner. It thus has an orbit of size $p'q'$, with $1 \leq q' \leq q$. Since this size is of the form $p^\alpha q^\beta$ (the group $\GG(X,r)$ having the size of this form), and since $p<q$ are primes, one necessarily has $p'=p$. Thus $\sigma_x$ permutes all the $p$ blocks in a cyclic manner. As a result, $x$ and $U(x)=\sigma_x^{-1}(x)$ lie in different blocks. Since $U=T^{-1}$, one obtains a $T$-cycle which does not entirely lie in a single block. Now, again by Theorem~\ref{T:cablingProperties}, one can choose a suitable $m$ such that the solution $(X,r^{(p^m)})$ is decomposable, with orbits whose sizes are multiples of $q$. The permutation group $\GG(X,r^{(p^m)})$ is a subgroup of the group $\GG(X,r)$ of size $p^a q^b$. Its size, as well as the sizes of all the $\GG(X,r^{(p^m)})$-orbits, are then of the same form. Being multiples of $q$, the sizes of the $\GG(X,r^{(p^m)})$-orbits are then all precisely $q$. One of them has to entirely contain our $T$-cycle of size $q$ (which is also a $T^{(p^m)}$-cycle). This $\GG(X,r^{(p^m)})$-orbit then intersects several blocks. Since the subgroup $\GG(X,r^{(p^m)})$ of $\GG(X,r)$ permutes these blocks, our $\GG(X,r^{(p^m)})$-orbit has to be of size $p'q'$, with $1<p' \leq p$ and $1 \leq q' \leq q$. But $q$ cannot be written in this way.
	
	\medskip	
	\noindent \textbf{Case 2:} There are $q$ blocks of size $p$.
	
	\noindent The permutations $\tau_x$ of $X$, for $x \in X$, generate a transitive group, since so do the $\sigma_x$, and  the two are related by the conjugation by~$T$ (see relation \eqref{E:sigma_tau}). Therefore some element $f \in X$ fixed by $T$ is moved to an element $c$ from a $T$-cycle of size $q$ by some $\tau_x$. That is, $c=\tau_x(f)$. We will use the relation
	\[T(\tau_x(f)) = \tau_{\sigma_{f}(x)}(f)\]
	from \cite[Lemma 3.8]{RV}. Applied $k$ times, it yields
	\[T^k(c) = \tau_{\sigma_{f}^k(x)}(f).\]	
	As a result, the size of the $\sigma_{f}$-orbit containing $x$ is a multiple of the size of the $T$-orbit containing $c$, which is $q$. Since $p<q$, it intersects  $q'>1$ blocks of size $p$. On the other hand, since $\sigma_{f}$ fixes $f$, it fixes at least one block. Thus $q'<q$. Summarizing, the size $q$ of our orbit decomposes as $p'q'$, with $1 \leq p' \leq p$ and $1 < q' < q$. But this is impossible.
	
	\begin{rem}
	Along the lines of the proof of \cite[Lemma 3.8]{RV}, one can establish the relation
	\[T^k(\tau_x(y)) = \tau_{\sigma_{ky}(x)}(T^k(y)),\]
	valid for all $x,y \in X$ (not necessarily $T$-fixed) and all $k$.
	\end{rem}

	\section{Applications: Dehornoy class}\label{S:DClass}
	
	In this section, we will prove Theorem~\ref{T:DClass}. The main ingredient is
	\begin{lem}\label{L:OrdersSameOrbit}
		For all elements $x$ from the same $\GG(X,r)$-orbit of a solution $(X,r)$, the order of $\sigma_x$ in the finite abelian group $(\GG(X,r),+)$ is the same.
	\end{lem}
	
	\begin{proof}
		Relation \eqref{E:cabling_R} specialised at $k=1$ yields the relation
		\begin{equation}\label{E:equivar}
		\sigma_x(ly) = l\sigma_x(y)
		\end{equation} in the structure group $G(X,r)$. In its quotient $\GG(X,r)$, it becomes \[
		\sigma_{\sigma_x}(l\sigma_y) = l\sigma_{\sigma_x(y)}.
		\]
		Thus $l\sigma_y$ vanishes if and only if $l\sigma_{\sigma_x(y)}$ does. As a consequence, $\sigma_y$ and $\sigma_{\sigma_x(y)}$ have the same order in $(\GG(X,r),+)$ for all $x,y \in X$.  
	\end{proof}
	
	\begin{rem}
		Relation~\eqref{E:equivar} means that the cabling operation $\iota^{(l)} \colon x \mapsto lx$ is equivariant with respect to the left $G(X,r)$-actions induced by the maps $\sigma_x$. It thus behaves better than the diagonal map~$T$, which instead of the equivariance obeys the less tractable rule \eqref{E:sigma_tau}.
	\end{rem}
	
	\begin{proof}[Proof of Theorem~\ref{T:DClass}]
		Relation~\eqref{E:class} can be rewritten as 
		\[\forall x \in X, \; \sigma_{x} \sigma_{U(x)}\cdots \sigma_{U^{m-1}(x)}=\Id,\]
		which, by~\eqref{E:kx}, simply means $m\sigma_{x}=0$. This yields the first assertion of the theorem. The second then directly follows from Lemma \ref{L:OrdersSameOrbit}.		
	\end{proof}
	
	We finish with the following observation, relating the Dehornoy class of a solution to its diagonal map:
	
	\begin{pro}
		Let $(X, r)$ be a solution. Then the order $|T|$ of its diagonal map divides its Dehornoy class $m$. 
	\end{pro}
	
	\begin{proof}
		We need to prove the relation $T^m=\Id$, or, equivalently, $U^m=\Id$. Let us compute 
		\begin{equation}\label{E:OrderVsDClass}
		r_{m,1}(x U(x) U^2(x) \cdots U^{m-1}(x), U^m(x))
		\end{equation}
		in two ways. On the one hand, the definition of the Dehornoy class allows one to simplify \eqref{E:OrderVsDClass} as
		\[r_{m,1}(mx, U^m(x)) = (\sigma_{mx}(U^m(x)), \cdot) = ((m\sigma_{x})(U^m(x)), \cdot) = (U^m(x), \cdot).\]
		On the other hand, since the tuple $x U(x) U^2(x) \cdots U^m(x)$ is frozen, \eqref{E:OrderVsDClass} equals
		\[(x, U(x) U^2(x) \cdots U^{m-1}(x)U^m(x)).\]
		Hence $U^m(x) = x$ for all $x \in X$.
	\end{proof}

	\subsection*{Acknowledgements}
	
	This work was partially supported by Conicet and 
	the project OZR3762 of Vrije Universiteit Brussel. The authors are grateful to Arpan Kanrar who suggested a stronger version of Theorem~\ref{T:2abc}, and to the reviewers for a thorough reading of the paper.
	
	\bibliographystyle{abbrv}
	\bibliography{refs}
	
\end{document}